\documentclass[11pt]{article}

\usepackage[letterpaper, left=3cm,top=2.5cm,right=3cm,bottom=3cm]{geometry}
\usepackage[utf8]{inputenc}

\usepackage{authblk}
\def\plus{{\boldsymbol{\texttt{+}}}}

\usepackage[onehalfspacing]{setspace}

\usepackage{booktabs}
\usepackage{epstopdf}
\usepackage[english]{babel}
\usepackage{amsmath,amsthm,amssymb}
\usepackage{dsfont}
\usepackage{mathtools}
\usepackage{subcaption}
\usepackage[color=green!40]{todonotes}
\usepackage{graphicx}
\usepackage{tikz}
\usepackage{bbm}
\usepackage{enumitem}
\setitemize{label=\scriptsize{$\blacksquare$},topsep=0em}
\setenumerate{label=(\alph*), topsep=0em}

\newcommand*{\N}{\mathds{N}}

\newcommand*{\R}{\mathds{R}}
\newcommand*{\C}{\mathds{C}}

\newcommand{\Bo}{\mathbf{B}}

\newcommand{\Do}{\mathbf D}

\newcommand{\Lo}{\mathbf L}
\newcommand{\Ao}{\mathbf A}

\newcommand{\So}{\mathbf S}

\newcommand{\X}{\mathbb X}
\newcommand{\Y}{\mathbb Y}

\newcommand{\nn}{\boldsymbol  \Phi}

\newcommand{\signal}{x}

\newcommand{\data}{y}

\newcommand{\reg}{\mathcal{R}}
\newcommand{\Qreg}{\mathcal{Q}}
\newcommand{\beg}{\mathcal{B}}
 
\newcommand{\fun}{\mathcal{F}}
\newcommand{\risk}{\mathcal{L}}

\DeclarePairedDelimiter{\abs}{\lvert}{\rvert}
\DeclarePairedDelimiter{\norm}{\lVert}{\rVert}
\DeclarePairedDelimiter{\inner}{\langle}{\rangle}

\DeclareMathOperator{\Id}{Id}
\DeclareMathOperator{\prox}{prox}

\DeclareMathOperator{\range}{ran}
\DeclareMathOperator{\dom}{dom}
\DeclareMathOperator{\ran}{ran}

\DeclareMathOperator*{\argmin}{arg\,min}

\newcommand{\al}{\alpha}
\newcommand{\la}{\lambda}

\DeclareGraphicsExtensions{.eps,.pdf,.png,.jpg}

\newtheorem{theorem}{Theorem}
\newtheorem{remark}[theorem]{Remark}
\newtheorem{definition}[theorem]{Definiton}

\newtheorem{assumption}[theorem]{Assumption}
\newtheorem{lemma}[theorem]{Lemma}
\numberwithin{equation}{section}
\numberwithin{theorem}{section}

\usepackage{soul}
\usepackage{xcolor}
\colorlet{lred}{red!40}
\colorlet{lgreen}{green!40}
\colorlet{lblue}{blue!40}

\newcommand{\kl}[1]{\left(#1\right)}
\newcommand\set[1]{{\left\{#1\right\}}}

\author{Markus Haltmeier\thanks{Corresponding author: \texttt{markus.haltmeier@uibk.ac.at}}}
\author{Richard Kowar}
\author{Markus Tiefenthaler}
 \date{}
 
\affil{Department of Mathematics, University of Innsbruck\authorcr
Technikerstrasse 13, 6020 Innsbruck, Austria}

\allowdisplaybreaks

\usepackage[numbers, sort&compress]{natbib}

\title{Data-driven Morozov regularization of inverse problems}

\date{\today}

\begin{document}

\maketitle

\begin{abstract}
The solution of inverse problems is crucial in various fields such as medicine, biology, and engineering, where one seeks to find a solution from noisy observations. These problems often exhibit non-uniqueness and ill-posedness, resulting in instability under noise with standard methods. To address this, regularization techniques have been developed to balance data fitting and prior information. Recently, data-driven variational regularization methods have emerged, mainly analyzed within the framework of Tikhonov regularization, termed Network Tikhonov (NETT). This paper introduces Morozov regularization combined with a learned regularizer, termed DD-Morozov regularization. Our approach employs neural networks to define non-convex regularizers tailored to training data, enabling a convergence analysis in the non-convex context with noise-dependent regularizers. We also propose a refined training strategy that improves adaptation to ill-posed problems compared to NETT's original strategy, which primarily focuses on addressing non-uniqueness. We present numerical results for  attenuation correction in photoacoustic tomography, comparing DD-Morozov regularization with NETT using the same trained regularizer, both with and without an additional total variation regularizer.


\end{abstract}

\section{Introduction} \label{sec:introduction}

In this paper, we consider the solution of linear inverse problems where we aim to reconstruct the unknown $\signal \in \X$ from noisy data
\begin{equation}\label{eq:ip}
	\data_\delta = \Ao \signal + \eta_\delta \,.
\end{equation}
Here $\Ao \colon \X \to \Y$ is a linear bounded operator between Hilbert spaces $\X$ and $\Y$, and $\eta_\delta $ is the data error that satisfies $\norm{\eta_\delta} \leq \delta$ with noise level  $\delta \geq 0$.  We are especially interested in the ill-posed case where solving \eqref{eq:ip} without prior information is non-unique or unstable. After discretization, this leads to linear systems of equations with perturbed right hand side and a forward matrix that possesses a large kernel, many small singular values, or both. Several applications in medical image reconstruction, nondestructive testing, and remote sensing are instances of such linear inverse problems \cite{EnHaNe96,NaWu01,ScGrGrHaLe09}. 

Characteristic features of inverse problems are the non-uniqueness of solutions and the unstable dependence of solutions on data perturbations.  To account for these two issues, one must apply regularization methods  (see, for example, \cite{ito2014inverse,EnHaNe96,ScGrGrHaLe09,benning2018modern,haltmeier2020regularization,schuster2012regularization,ivanov2013theory,morozov1993regularization,tikhonov1977solutions}) that serve two main purposes: First, in the case of exact data $\data \in \range(\Ao)$, they select a specific solution $\Bo_0 (\data)$ among all possible solutions of the exact data equation $\data = \Ao x$.  Second, to account for noise, they define stable approximations to $\Bo_0$ in the form of continuous mappings $\Bo_\alpha \colon \Y \to \X$ that converge to $\Bo_0$ as $\al \to 0$ in an appropriate sense.  

\subsection{Morozov regularization}

There are several well established methods  for the stable solution of inverse problems. A general class of regularization  methods are variational regularization methods which includes Tikhhonov regularization, Ivanov regularization (the method of quasi solutions), and  Morozov regularization (the residual method) as special cases. In Tikhonov regularization, approximate solutions are defined  as minimizers of  
$ \norm{\Ao \signal- \data_\delta}^2/2  + \alpha \reg(\signal)$, where   $\reg \colon \X \to [0, \infty]$ is a regularization functional that measures the feasibility of a potential solution and $\alpha$ is the regularization parameter. Ivanov regularization considers minimizers of $\norm{\Ao \signal- \data_\delta}  $ over the set   $\{   \signal \in \X \mid  \reg(\signal) \leq \tau \}$ for some  $\tau >0$.   In this paper, we consider Morozov regularization where approximate solutions defined as solutions of   
\begin{equation}\label{eq:morozov}
     \min_{\signal \in \X} \reg(\signal) \quad  \text{s.t. }\quad  \norm{\Ao \signal- \data_\delta} \leq \delta\,.
\end{equation}
Compared to Tikhonov regularization  and Ivanov regularization, the latter has the advantage that no additional regularization parameter has to be selected, which is typically a difficult issue. Relations between Tikhonov regularization, Ivanov regularization and Morozov regularization are carefully studied in~\cite{ivanov2013theory}. A general convergence analysis of Morozov regularization, including convergence rates, has been carried out in \cite{grasmair2011residual}.

Note that variational regularization methods are designed to approximate $\reg$-minimizing solution of $\Ao \signal = \data $ for the limit $\delta \to 0$, defined as elements  in $ \argmin  \{ \reg(\signal) \mid  \Ao \signal = \data\}  $. This addresses the non-uniqueness in the case of exact data.  To account  for   noisy data $\data_\delta$,  Morozov regularization  relaxes the strict data consistency $\Ao \signal = \data$  to data proximity $\norm{\Ao \signal- \data_\delta} \leq \delta$.  Even in the case that $\Ao$ is injective, different  regularization terms behave differently and significantly affect  convergence. Therefore the choice of the regularizer is crucial and a nontrivial issue. Classical choices for the
regularizers are the squared Hilbert space norm $\reg(\signal) =  \norm{\signal}^2/2$ or the
$\ell^1$-norm   $\reg(\signal) = \sum_{\la \in \Lambda} \abs{\inner{u_\la, \signal}}$, where $(u_\la)_\la$ is a frame of $\X$. These regularizers may not be optimally adapted to highly structured signal classes, as is often the case in practical applications.  In this paper, we address this issue and propose a data-driven regularizer using neural networks adapted to the signal class represented by training data in combination with Morozov regularization.

\subsection{Neural network regularizers}

In this paper, we study Morozov regularization with a neural network based,  data-driven and noise-level-dependent regularizer  
\begin{equation} \label{eq:reg1}
	\reg_\delta ( \signal )
            = \frac{1}{2} \norm{ \nn_{\theta(\delta)} (\signal) - \signal}^2
                  + \lambda \,\Qreg ( \signal )  \,.
\end{equation}   
Here $\nn_{\theta(\delta)} \colon \X \to \X$ is a neural network tuned to noisy data, and $\Qreg \colon \X \to [0, \infty]$ is an additional regularization term.  While $\Qreg$ is mainly added for theoretical purposes to guarantee coercivity, in the numerical results we show that it also has a beneficial influence on the reconstruction.  For the theoretical analysis the squared   Hilbert  space norm $\Qreg = \norm{\cdot}^2$ would be sufficient. We refer to \eqref{eq:morozov} with the data-driven regularizer \eqref{eq:reg1} instead of the fixed regularizer $\reg$ as  data-driven Morozov (DD-Morozov)  regularization. We will show that, under reasonable assumptions, the latter provides a convergent regularization method. Furthermore, we present a training strategy for selecting the noise-dependent neural network for a given architecture.  Note that a noise-level-independent network is also included in our analysis. Allowing the networks to depend on $\delta$, however, enables the networks to better adapt to available noisy training data.

Besides stabilizing the signal reconstruction, the main purpose of a particular regularizer is to fit the reconstructions to a certain set where the true signals are likely to be contained. In reality, this set is not known analytically, but it is possible to draw examples from it. For this reason, we follow the learning paradigm and choose training signals $\signal_1, \dots, \signal_N \in \X $ and adapt the architecture $(\nn_\theta)_{\theta \in \theta}$ to $\signal_i$ and corresponding noisy measurements . More precisely, $\theta = \theta(\delta)$ is chosen such that $\nn_\theta( \signal_i ) \simeq \signal_i $ and $\nn_\theta( \signal_i + r_{i,j}) \simeq \signal_i $, where $r_{i,j}$ are certain perturbations that are allowed to depend on the forward operator and on the the noise. Thus, $\reg_\delta$ has small values for the exact $\signal_i $ and larger values for the perturbed signals $\signal_i + r_{i,j}$. Since the training signals $\signal_i$ are only taken from a certain potentially small subset of $\X$, it is difficult to obtain the necessary coercive condition from training alone. Therefore, it seems natural to add another regularization term $\Qreg$ which is known to be coercive. 

In Section~\ref{sec:train} we propose a possible choice for the   perturbations $r_{i,j}$ resulting in a novel training strategy for the regularizer.  This refines  the  training strategy of the NETT paper \cite{li2018nett} and seems  more suitable for ill-posed and ill-conditioned problems. In the original NETT strategy, all added perturbations are in the null space of the forward operator, and the task of the learned regularizer was purely the selection of a proper solution in the limit. The refined training strategy refers to the added perturbations $r_{i,j}$ that might have components outside the null space of the forward operator. In the case of several small singular values the proposed  regularizer can correct for signal components damped by  small singular values with an  amount determined by the noise level.

While learned regularizers have recently become  popular in the context of Tikhonov regularization \cite{li2018nett,obmann2021augmented,lunz2018adversarial,goujon2023neural}, we are not aware of any work utilizing the Morozov variant.  In fact, our analysis as well as the training strategy are related to the Network Tikhonov Approach (NETT) of \cite{li2018nett,obmann2021augmented}. A very different strategy for learning a network regularizer has been proposed in \cite{lunz2018adversarial} in the context of adversarial regularization. Other approaches for learning a regularizer are the fields of experts model \cite{roth2005fields}, deep total variation \cite{kobler2020total} or ridge regularizers \cite{goujon2023neural}.    Other data-driven regularization methods for inverse problems can be found for example in \cite{arridge2019solving,aspri2020data,aspri2020dataB,haltmeier2020regularization,mukherjee2021end,riccio2022regularization,schwab2019deep,dittmer2020regularization} and the references therein.   From the theoretical side, Morozov regularization in a general non-convex context has been studied in \cite{grasmair2011residual}. The analysis we present below allows the regularizer to be noise-dependent and further we derive strong convergence under  total nonlinearity condition of~\cite{li2018nett}.

\subsection{Outline}

The remainder of this paper is organized as follows. In Section~\ref{sec:theory} we present our theoretical results. In particular we present the convergence analysis (Section~\ref{sec:analysis}) and the prosed training strategy (Section~\ref{sec:train}). In Section~\ref{sec:results} we present numerical results illustrating DD-Morozov regularization  with and without TV as additional regularizer.   Specifically, we test our approach  on an inverse problem for one-dimensional  attenuation correction in photoacoustic tomography in damping media \cite{kowar2011attenuation} and compare DD-Morozov regularization  with and without TV to NETT regularization and pure TV regularization.  To numerically solve  \eqref{eq:morozov} we implement the primal dual scheme of \cite{condat2013primal}.   The paper concludes with a short summary in Section~\ref{sec:conclusion}.

\section{Theory}
\label{sec:theory}

Throughout this paper, $\X$, $\Y$  are Hilbert spaces and $\Ao \colon \X \to \Y$ a bounded linear operator  with potentially non-trivial nullspace. Recall that a functional $\reg \colon \X \to [0, \infty]$ is coercive, if $\reg(\signal_n) \to  \infty$ for all sequences
$(\signal_n)_{n\in\N} \in \X^\N$ with $\|\signal_n\|_\X  \to \infty$, and weakly lower semicontinuous, if $\reg(\signal) \leq \liminf_{n\to \infty} \reg(\signal_n)$ for 
$(\signal_n)_{n\in\N} \rightharpoonup \signal$, where $\rightharpoonup$  denotes weak convergence, and $\to$ strong convergence. Any element in $\argmin \set{\reg(\signal) \mid  \Ao \signal  = \data }$ is called an $\reg$-minimizing solution of the equation $\Ao\signal = \data$. 

\subsection{Convergence analysis}
\label{sec:analysis}

For neural networks $\nn, \nn_{\theta(\delta)}$ on $\X$ and noise level $\delta >0$ we define the noise-dependent regularizer $\reg_\delta$ by \eqref{eq:reg1},  the limiting regularizer by $\reg ( \signal ) = \norm{ \nn (\signal) - \signal}^2/2 + \lambda \,\Qreg ( \signal )$ and consider DD-Morozov regularization
\begin{equation}\label{eq:morozov1}
     \min_{\signal \in \X} \reg_\delta(\signal) \quad  \text{s.t. }\quad  \norm{\Ao \signal- \data_\delta} \leq \delta \,.
\end{equation}
Our  results on the convergence of  \eqref{eq:reg1}, \eqref{eq:morozov1} are derived under the following conditions, which we assume to be satisfied throughout this subsection.

\begin{assumption}\mbox{}\label{ass:well}
\begin{enumerate}[label=(A\arabic*), leftmargin=3em, topsep=0em, itemsep=0em]
\item \label{a1} $\nn,  \nn_{\theta(\delta)} \colon \X \to \X$ are weakly continuous.
\item \label{a2}  $\nn_{\theta(\delta)}  \to \nn $ weakly uniformly on bounded  sets as $\delta \to 0$. 
\item \label{a2a}  $\nn_{\theta(\delta)}  \to \nn $ strongly pointwise on  $\reg$-minimizing solutions as $\delta \to 0$. 
\item \label{a3} $\Qreg \colon \X \to [0, \infty]$ is proper, coercive and weakly lower semicontinuous.    
\end{enumerate}
\end{assumption}

In \ref{a2},  weak uniform convergence  on bounded  sets  means that for  all bounded  $B \subseteq \X$ and all $h \in \X$ we have $\sup_{\signal  \in B} \abs{\inner{ \nn_{\theta(\delta)}(\signal)  - \nn (\signal),h}} =0$ as $\delta \to 0$.   In the convergence analysis we assume that the networks $\nn_{\theta(\delta)}$ are trained, where $\theta(\delta)$ potentially depends on the noise level, and that the forward operator  $\Ao$  and the additional regularizer $\Qreg$ are given by the application or are user-specified. In many applications, $\norm{
(\nn_{\theta(\delta)} - \Id) (\cdot) }$ may not be coercive which is the reason to add the term  $\Qreg$ in \eqref{eq:reg1}.

\begin{lemma}\label{lem:R}
The regularizers  $\reg, \reg_\delta$ are coercive and weakly sequentially lower semicontinuous. Further, the feasible set $\{\signal \in \X \mid \norm{\Ao \signal- \data_\delta}  \leq \delta \}$ is weakly closed and non-empty for all $\delta >0$ and all data $\data_\delta$ with $\norm{\Ao \signal_\star - \data_\delta } \leq \delta $ for some $\signal_\star \in \dom(\Qreg)$. 
\end{lemma}

\begin{proof}
 Let $(\signal_n)_{n\in\N} \in \X^\N$.
Because $\Qreg$ is coercive and $\norm{ \nn (\signal) - \signal}^2$ , $ \norm{ \nn_{\theta(\delta)} (\signal) - \signal}^2$  are non-negative,  the functionals   $\reg$, $\reg_\delta$  are  coercive. Let $(\signal_n)_{n\in\N} \in \X^\N$ converge weakly to  $\signal \in \X$. Because  $\nn$  is  weakly continuous, 
$( \nn(\signal_n) - \signal_n)_{n\in\N}$ converges weakly to $\nn(\signal) - \signal$. Due to the  weak  sequential lower semicontinuity of the norm, we infer $  \| \nn(\signal) - \signal \|_\X
     \leq \liminf_{n\to\infty} \| \nn(\signal_n) - \signal_n \|_\X$ which shows that  $ \reg$ and in a similar manner $ \reg_\delta$ are weakly lower semicontinuous.  Now, according  to~\cite[Lemma~1.2.3]{BlBr92}, a functional  $\fun$ is weakly sequentially lower semicountinuous  if and only if   $ \{ \signal\in\X \mid  \fun(\signal) \leq t \}$ is weakly sequentially closed for all  $t>0$. Because $\Ao$ is linear and bounded it is  weakly continuous.  Because the norm is weakly sequentially lower semicontinuous,  $\signal  \mapsto \norm{\Ao \signal - \data_\delta} $ is weakly sequentially lower semicontinuous, too.  Hence  $\{\signal \in \X \mid \norm{\Ao \signal- \data_\delta}  \leq \delta \}$ is weakly closed for all $\delta >0$ and non-empty as it contains  the exact data $\Ao \signal_\star$.
\end{proof}

\begin{lemma}[Existence]\label{lem:existence}
For all data $\data_\delta \in \Y$ with $\norm{\Ao \signal - \data_\delta } \leq \delta $ for some $\signal \in \dom(\Qreg)$, the constraint optimization problem~\eqref{eq:morozov1} has at least one solution.    
\end{lemma}

\begin{proof}
Because $\reg_\delta \geq 0$, the infimum $M$ of $\reg$ over $ S_\delta \coloneqq \{\signal \in \X \mid \norm{\Ao \signal- \data_\delta}  \leq \delta \}$  is nonnegative and there exists a sequence $(\signal_m)_m$ of elements of $S_\delta$ with $ \lim_{m\to\infty} \reg_\delta(\signal_m) = M$.   Because $\reg_\delta$ is coercive and $\{ \reg_\delta(\signal_m)\,|\,m\in\N\}$ is bounded, we infer that
$(\signal_m)_m$ is bounded and thus there  exists a weakly convergent  subsequence
$(\signal_{m(k)})_k$ converging to some $\signal \in \X$. Moreover, due to the weak closedness of $S_\delta$  we  obtain $\signal\in S_\delta$.  Because $\reg_\delta$  is weakly lower semicontinuous, $\reg(\signal) \leq  M$ and thus $\signal$  is a solution of \eqref{eq:morozov1}. 
\end{proof}

Analogous to the proof of Lemma~\ref{lem:existence} one shows that  there exists  at least one $\reg$-minimizing solution  of  $\Ao\signal = \data$ whenever it  is  solvable in $\dom(\Qreg)$.

\begin{theorem}[Weak convergence]\label{thm:weak}
Let $\Ao\signal = \data$ be solvable in $\dom(\Qreg)$, $(\data_n)_{n\in\N} \in \Y^\N$ satisfy $\norm{ \data - \data_n} \leq \delta_n$, where
$(\delta_n)_{n\in\N} \in (0, \infty)^\N$ with $\delta_n \to 0$, write $\reg_n \coloneqq \reg_{\delta_n} $ and $\nn_n \coloneqq \nn_{\theta(\delta_n)}$, and choose $ \signal_n \in \argmin \{\reg_n (z) \mid \norm{ \Ao z - \data_n} \leq \delta_n \}$.  Then  $(\signal_n)_{n\in\N}$ has at least one weak accumulation point $\signal^\plus\in \X$.
Moreover, the limit of each weakly converging subsequence $(\signal_{n(k)})_{k\in\N}$ is an $\reg$-minimizing solution of $\Ao \signal = \data$ and  $\reg_{n(k)} (\signal_{n(k)}) \to \reg(\signal^\plus)$ for $k\to\infty$. If the $\reg$-minimizing solution $\signal^\plus$ of $\Ao \signal = \data$ is unique, then $\signal_n \rightharpoonup \signal^\plus$ and $\reg_n (\signal_n) \to \reg(\signal^\plus)$ as $n\to\infty$.
\end{theorem}

\begin{proof}
Set $S_n \coloneqq  \{z \in \X \mid \norm{\Ao z - \data_n}  \leq \delta_n\}$ and $S_\star \coloneqq \{z \in \X \mid  \Ao z = \data \}$. Clearly  $S_\star \subseteq S_n$  and because  $\Ao\signal = \data$ is solvable, $S_\star$ is non-empty.   Thus $ \Qreg (\signal_n) +  
 \norm{\nn_n(\signal_n)-\signal_n}^2/2 =
 \reg_n (\signal_n)  \leq  \reg_n (\signal_\star)  =  \Qreg (\signal_\star) +  
 \norm{\nn_n(\signal_\star)-\signal_\star}^2/2$, where $\signal_\star$ is an $\reg$-minimizing solution of $\Ao \signal =\data$.  
 From the weak convergence  of $\nn_n(\signal_\star)$ we see that the right hand side is bounded. Thus $ \Qreg (\signal_n)$ is bounded and with the coercivity  of $ \Qreg$ we conclude there exists a weakly converging subsequence $(\signal_{n(k)})_{k \in \N}$.  Because $ \|\Ao(\cdot) - \data\|$ is  weakly lower semicontinuous,  
\begin{multline} \label{eq:conv1}
    \|\Ao(\signal^\plus) - \data\| 
      \leq \liminf_{k\to\infty}  \| \Ao (\signal_{n(k)}) - \data\|   
      \\ \leq \liminf_{k\to\infty}  \| \Ao (\signal_{n(k)}) - \data_{n(k)}\|  + \|\data_{n(k)} - \data\|  
      \leq 2 \delta_{n(k)} 
 \end{multline}
and thus  $\signal^\plus\in S_\star$. It remains to verify that   $\signal^\plus$  is an $\reg$-minimizing solution of $\Ao \signal =  \data$. According to  \ref{a1}, \ref{a2}  we have $\nn_{n(k)}(\signal_{n(k)})  \rightharpoonup \nn(\signal^\plus)$ and  $\nn_n (\signal_\star)  \to \nn (\signal_\star) $ as the  weak  uniform convergence assumption \ref{a2} implies that  for all bounded $B \subseteq \X$ and all  $z \in \X$ we have  $\sup_{x \in B} \inner{\nn_{n(k)}(x)-\nn(x), z} \to 0$. Thus
\begin{align*} 
       \Qreg(\signal^\plus) + \frac{\la}{2} \norm{\nn (\signal^\plus) -\signal^\plus }
       &\leq 
       \liminf_{k \to \infty} \Qreg(\signal_{n(k)}) + 
       \frac{\la}{2}
       \norm{\nn_{n(k)} (\signal_{n(k)}) -\signal_{n(k)} }
       \\& \leq 
       \limsup_{k \to \infty}
       \Qreg(\signal_{n(k)}) + \frac{\la}{2} \norm{\nn_{n(k)} (\signal_{n(k)}) -\signal_{n(k)} } 
       \\& \leq 
       \limsup_{k \to \infty}
       \Qreg(\signal_\star) + \frac{\la}{2} \norm{\nn_{n(k)} (\signal_\star) - \signal_\star } 
       \\& = 
       \Qreg(\signal_\star) + \frac{\la}{2} \norm{\nn (\signal_\star) - \signal_\star }  \,.
\end{align*}
Therefore $\signal^\plus$ is   an $\reg$-minimizing solution of $\Ao \signal = \data$ with  $\reg (\signal_{n(k)}) \to \reg(\signal^\plus)$. Finally, if  the $\reg$-minimizing solution $\signal^\plus$ of $\Ao \signal = \data$ is unique, then $(\signal_n)_{n \in \N}$ has exactly one weak accumulation point  $ \signal^\plus$ and $\reg_n (\signal_n) \to \reg(\signal^\plus)$.
\end{proof}

Following \cite{li2018nett}, we introduce the concept of total nonlinearity, which is required for strong convergence.

\begin{definition}[Total nonlinearity]
Let $\fun \colon \X\to\R$ be G\^ateaux differentiable at $\signal_\star\in\X$. The absolute Bregman distance $\beg_\fun(\signal_\star, \cdot) \colon \X\to[0,\infty]$  and the modulus of total nonlinearity $\nu_\fun(\signal_\star,\cdot):(0,\infty)\to[0,\infty]$ of $\fun$ at $\signal_\star$  are defined by   
\begin{align} \label{eq:bregman}
    \forall \signal \in\X \colon & \quad
    \beg_\fun(\signal_\star,\signal) \coloneqq  |\fun (\signal) - \fun(\signal_\star) - \fun'(\signal_\star)(\signal-\signal_\star)|  \\
    \forall t>0  \colon & \quad    
    \nu_\fun(\signal_\star,t) \coloneqq  \inf \{ \beg_\fun(\signal_\star,\signal) \,|\, \signal \in\X \wedge  \| \signal - \signal_\star\|_\X = t \} \,.
\end{align}
The functional $\fun$ is called totally nonlinear at $\signal_\star$, if $\nu_\fun(\signal_\star, t)>0$ for all $t\in (0,\infty)$.
\end{definition}

According to \cite{li2018nett}, $\fun$  is  totally nonlinear at $\signal_\star \in\X$ if and only if for all bounded sequences $(\signal_n)_{n\in\N} \in \X^\N$ with $ \lim_{n\to \infty} \beg_\fun(\signal_\star,\signal_n) = 0  $ we have  $\lim_{n\to\infty} \|\signal_n-\signal_\star \|_\X = 0$.

\begin{theorem}[Stong convergence] \label{thm:strong}
In the  situation of Theorem~\ref{thm:weak} assume additionally that  the $\reg$-minimizing solution $\signal^\plus$ of $\Ao \signal = \data$ is unique and that $\reg$ is totally nonlinear at $\signal^\plus$. Then, $ \|\signal_n - \signal^\plus\|_\X \to 0$ as $n\to\infty$.
\end{theorem}

\begin{proof}
According to Theorem~\ref{thm:weak}, the sequence $(\signal_n)_{n\in\N}$ converges weakly to $\signal^\plus$ and $\reg(\signal^\plus) =  \lim_{n\to\infty} \reg(\signal_n)$. Because $\reg'(\signal^\plus)$ is bounded, $ \reg'(\signal^\plus)(\signal_n - \signal^\plus)  \to 0$ and thus 
$ \beg_{\reg}(\signal^\plus,\signal_n) =  |\reg(\signal_n) - \reg(\signal^\plus)
                     - \reg'(\signal^\plus)(\signal_n - \signal^\plus)| \to  0$. Because  $(\signal_n)_{n\in\N}$ is bounded with the  total nonlinearity of $\reg$ this yields  $\signal_n  \to  \signal^\plus$.
\end{proof}

\subsection{Training strategy}
\label{sec:train}

Given a sequence of noise levels $(\delta_n)_{n\in \N}$, our aim is  to construct the data-driven regularizer $ \norm{\nn_n (\signal) - \signal }^2/2$ with neural networks $\nn_n$ adapted to training signals $\signal_i \in \X$ for $ i \in I \coloneqq \set{1, \dots, N}$ that we consider as ground truth and  corresponding perturbed signals $\signal_i + r_{i,n,j} \in \X$ for $j \in J_{i,n}$ that we want to avoid. Given a family  $(\nn_{\theta})_{\theta \in \Theta}$,  we determine the parameter $\theta = \theta_n$  as the minimizer of          
\begin{equation} \label{eq:risk}
    \risk_n( \theta ) \coloneqq  \sum_{i\in I} \sum_{ j \in J_{i,n}^\star} \norm{ \nn_\theta(\signal_i + r_{i,n,j}) - \signal_i }^2    \,,
\end{equation}
where $J_{i,n}^\star  \coloneqq J_{i,n} \cup \{0\} $ and $r_{i,n,0} \coloneqq 0$ for the ground truth signals.  Notice that the index $i$ indicates the training example, the index $n$ the noise level, and the index $j$ refers to the perturbations for a given training example and noise level.

By doing so, we have  $(\nn_n - \Id) ( \signal_i )   \simeq  0$ for the ground truth signals $x_i$ and  $(\nn_n - \Id) ( \signal_i+ r_{i,n,j} ) \simeq r_{i,n,j} $ for the perturbed signals $\signal_i + r_{i,n,j}$.  Hence the regularizer $ \norm{ (\nn_n - \Id) ( \cdot ) }  $ is expected to be small for signals similar to $\signal_i$ and large for signals  similar to $\signal_i + r_{i,n,j}$.  A specific feature of  our learned regularizer is that it can depend on the forward problem. This is achieved by making the perturbations $r_{i,n,j}$ operator specific. A strategy for  increasing  this dependence is to let the architecture depend on $\Ao$  such as a null space  network  \cite{schwab2019deep} or data-proximal  network \cite{goppel2023data}.

\begin{remark}[Choice of the perturbations] \label{rem:perturbations}
A crucial question is how to construct proper perturbed signals $\signal_i + r_{i,n,j} \in \X$ for $j \neq 0$. For NETT, we proposed in \cite{li2018nett} to choose a single perturbation $r_{i,n,1} = r_i^{\ast} := (\Id - \Ao^\plus \Ao)(\signal_i) \in \ker(\Ao)$ per training example, which is also independent of the noise. This choice is well suited to address non-uniqueness, a primary issue in undersampled tomographic inverse problems where the kernel $\ker(\Ao)$ has high dimensionality. In this work, we are also concerned with ill-posed and ill-conditioned problems where small singular values pose an additional challenge, alongside the possibly large kernel. As a result, we modify the training strategy of \cite{li2018nett} by introducing multiple perturbations $r_{i,n,j}$ that address two additional factors: some perturbations represent noise in the low-frequency components corresponding to large singular values, while others reflect damped high-frequency components of the signal corresponding to small or vanishing singular values.
\end{remark}

Let $(u_n, v_n, \sigma_n)_{n \in \N}$ denote a singular value decomposition (SVD) of $\Ao$. We allow the operator $\Ao$ to have a nontrivial kernel $\ker(\Ao) \neq \set{0}$ and a non-dense range $\overline{\ran(\Ao)} \neq \Y$, and we only include the non-vanishing singular values in the definition of $\sigma_n$. In such a situation we can express $\Ao$, its pseudoinverse, and the truncated SVD reconstruction by the formulas
 \begin{align*}
     \Ao (\signal) &= \sum_{n\in \N} \sigma_n \inner{\signal,u_n} v_n
     \\
     \Ao^\plus (y) &=  \sum_{n\in \N}  \sigma_n^{-1}   \langle \data, v_n\rangle u_n 
     \\
     \So_\alpha (y) &=  
     \sum_{\sigma_n^2 \geq \alpha} \sigma_n^{-1}   \langle \data, v_n\rangle u_n \,,
\end{align*}
where $\alpha\geq 0$  is the regularization parameter. 

Now if $\signal_i$ is a given ground  truth  signal  and $\data_{i,n} = \Ao \signal_i + z_{i,n}$ corresponding noisy data, and $\al[n,j]$ for $j \in J_n$ are variable chosen regularization parameters in truncated SVD, we consider perturbed signals    
\begin{equation*}
	\signal_i  + r_{i,n,j} \coloneqq  
	\So_{\alpha[n,j]} ( \Ao \signal_i + z_{i,n}) \,.
\end{equation*} 
In fact, the perturbed signals are truncated SVD regularized reconstructions with perturbations $r_{i,n,j} = \So_{\alpha[n,j]} ( \Ao \signal_i + z_{i,n} ) - \signal_i$. In particular, for $z_{i,n} = 0$ and $\alpha[n,1] =0$, we recover the perturbations $r_i^{\ast} = \Ao^\plus \Ao \signal_i - \signal_i$ proposed in \cite{li2018nett} for NETT, which can be seen as pure artifacts (meaning elements in the null space caused by projecting $\signal_i$ onto $\ker(\Ao)^\bot$). The more general strategy that is proposed here also includes perturbations due to noise and to Gibbs-type artifacts caused by truncation of singular components.

\section{Application}
\label{sec:results}

In this section, we present numerical results for an inverse problem in attenuation correction in photoacoustic tomography (PAT), as described in more detail below. We consider a discrete setting where the operator $\Ao \in \R^{d \times d}$ is a matrix of size $d = 601$ and $\signal \in \R^d$ is the time discretization of a real-valued function defined on the interval $[0,T]$. The additional regularizer $\Qreg$ is taken either as zero or as the total variation (TV). Details on the forward operator, the learned regularizer, and the numerical solution of \eqref{eq:morozov1} are given below.

\subsection{Attenuation correction in PAT using the NSW model}

The problem we have selected for testing the learned Morozov regularization is attenuation correction in PAT. The  corresponding forward  operator has a significant number of large and small singular values, making appropriate regularization highly important. It is 1D and thus numerically efficient and of non-convolutional form, ensuring it is not too simplistic to draw conclusions from. Finally, it is of practical significance, as 1D attenuation correction can serve as a module for two- and three-dimensional image reconstruction, as we argue below.

\paragraph{PAT in attenuating media:} PAT is a hybrid imaging  technology  based on the conversion of optical energy into  acoustic pressure waves, and  combines the high spatial resolution of ultrasound imaging with the high contrast of optical imaging. A general model for PAT in attenuating media is (see \cite{kowar2012photoacoustic,kowar2011attenuation}) 
\begin{equation}  \label{eq:wavealpha}
	\kl{ \Do_\al  +
	\frac{1}{c_0}\frac{\partial}{\partial t }  }^2  p_\al(x,t) -
    \Delta p_\al(x,t)    =  \delta'(t) h(x)
	\quad \text{ for }(x,t) \in \R^D \times \R \,.
\end{equation}
Here  $D \in \set{1,2,3}$ is the  spatial dimension, $h \colon \R^D \to \R$ is the photoacoustic (PA) source, $c_0 > 0$ is a constant, and $\Do_\al$ is the time convolution operator  associated with the inverse Fourier transform of the complex valued attenuation function $\al \colon \R \to \C$.
The image reconstruction problem  associated to \eqref{eq:wavealpha} consists in reconstructing the source $h$ from measurements of $p_\al$ recorded at detector positions  $x$ on an $(D-1)$-dimensional manifold. The  case $\al = 0$ corresponds  to absence of attenuation where several efficient and accurate inversion methods exist including explicit inversion formulas~\cite{finch2007inversion,finch2004determining,haltmeier2015universal,
kunyansky2007explicit,natterer2012photo,palamodov2012uniform,xu2005universal} or iterative approaches
\cite{deanben2012accurate,paltauf2002iterative,rosenthal2013current,arridge2016adjoint,belhachmi2016direct,haltmeier2016iterative,huang2013full}.

\paragraph{The NSW model:}  
A variety of attenuation models exist, depending on the form of the complex-valued function $\alpha$; causal models are of particular interest \cite{kowar2012photoacoustic,kowar2011attenuation}. In this paper, we work with the causal model of Nachman, Smith, and Waag (NSW model) with a single relaxation process \cite{nachman1990equation}, for which the complex attenuation function takes the form
\begin{equation} \label{eq:nsw}
   \alpha(\omega) =  \frac{(-i\,\omega)}{c_\infty}\,\left(
      \frac{c_\infty}{c_0}\,\sqrt{  \frac{1 + ({c_0}/{c_\infty})^2\,(-i\,\tau_1\,\omega)}{1 + (-i\,\tau_1\,\omega) }}   -1 \right) 
      \,,
\end{equation}
with  parameters $c_0$, $\tau_1$, $c_\infty$. Equation \eqref{eq:nsw} and its generalization using $N$ relaxation processes have been derived in~\cite{nachman1990equation} based on sound physical principles.

\paragraph{Attenuation correction:} 
According to \cite[Theorem~1 and Lemma~1]{kowar2010integral}, the attenuated pressure signals can be expressed in terms of the unattenuated ones by
\begin{align}  \label{eq:Mpalpha}
p_\alpha(x,t)
&= \int_0^t  m_\alpha(t,r) p_0(x,r) \, dr \,, 
\\ \label{eq:Mpalpha2}
\mathcal{F} (m_\alpha  (\cdot,r)  )(\omega)
&= \frac{\omega}{ \omega/c_0 + i \alpha(\omega)}
\, e^{i (\omega / c_0  + i \alpha(\omega)) \abs{r}} \,,
\end{align}
for $r,t > 0$, $\omega \in \R$ and $x \in \R^D$. Here, $p_\alpha$ and $p_0$ denote the causal solutions of \eqref{eq:wavealpha} with and without attenuation, respectively, and $\mathcal{F}$ is the Fourier transform in the first argument. This shows that the PA source $h$ can be recovered by first solving the 1D reconstruction problem \eqref{eq:Mpalpha} for $p_0$ and then recovering $h$ from the estimated unattenuated signals. Such a two-step method has been proposed and implemented for the power law in~\cite{lariviere2005image,lariviere2006image}, and later used in~\cite{ammari2012photoacoustic,kowar2012photoacoustic} for various attenuation laws. We refer to the inversion of \eqref{eq:Mpalpha} as 1D attenuation correction.

\paragraph{Note on the ill-posedness:} 
The analysis of \cite{elbau2017singular} on the singular values of PAT in attenuating media shows that reconstructing $h$ based on the NSW model shares the same asymptotic behavior of singular values as the inversion problem in PAT without attenuation. Consequently, in the continuous setting, the 1D attenuation correction problem defined by \eqref{eq:Mpalpha} is expected to be stable. However, the singular values are predicted to decay rapidly at first, as the attenuation of a pressure-emitting source increases with both $\omega$ and the distance from the observer. The singular value plot for the discretized operator in Figure \ref{fig:A} confirms this expectation and demonstrates that the discrete problem is ill-conditioned.

\begin{figure}[htb!]
\begin{center}
\includegraphics[height=0.23\textwidth]{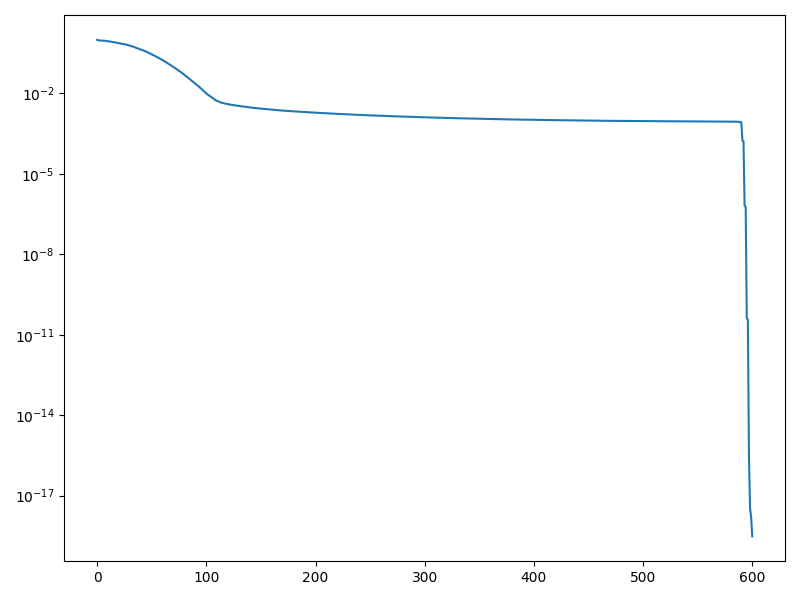}
\includegraphics[height=0.23\textwidth]{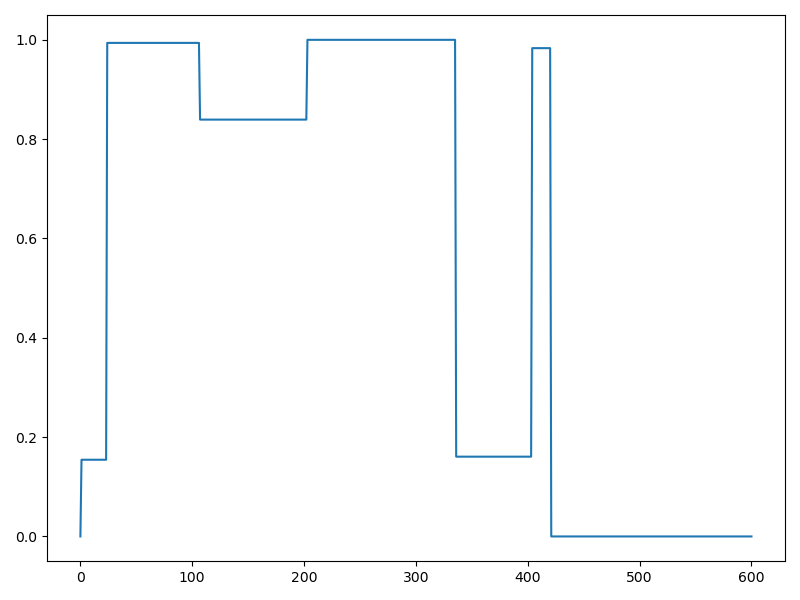}
\includegraphics[height=0.23\textwidth]{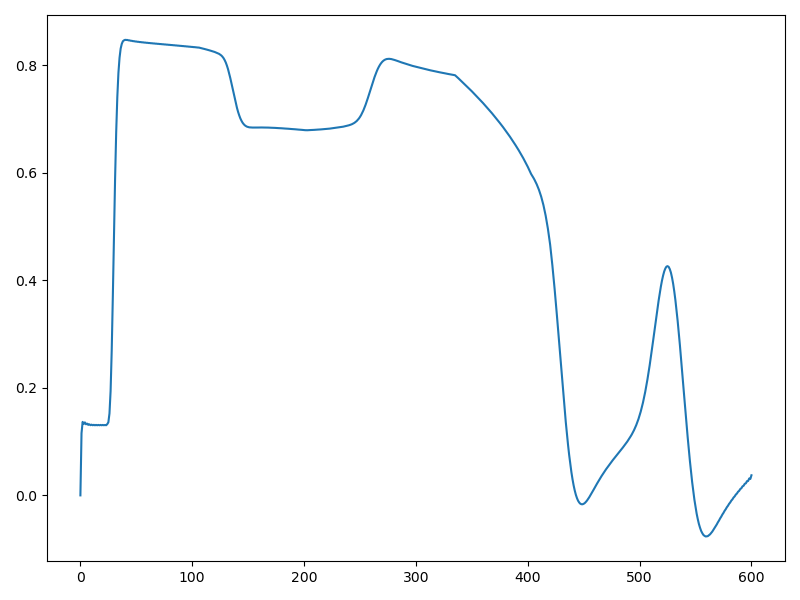}
\caption{Left: Singular values of  the $\Ao  \in\R^{d \times d}$ modeling dissipation  on a logarithmic scale. Middle: Test signal $\signal \in \R^d$. Right: Exact data   $\Ao \signal$ for input  from  the middle picture. The horizontal axis in the middle and right images represents time.}
\label{fig:A}
\end{center}
\end{figure}

\subsection{Implementation details}

\paragraph{Forward operator:} 
The forward operator \(\Ao\) is taken as a discretization of the 1D integral operator defined by \eqref{eq:Mpalpha} that maps unattenuated pressure signals to attenuated signals.  We use the NSW model~\eqref{eq:nsw} with parameters \(c_0 = 1\), \(\tau_1 = 10^{-4}\), and \(c_\infty = 1.41\) calculated over the temporal interval \([0, 0.1]\). Its operation is illustrated in Figure~\ref{fig:A} (center and right). For details on how the matrix \(\Ao\) was computed, see \cite{HaKoNg17}. Due to the fast decay of the singular values of \(\Ao\) (Figure~\ref{fig:A}, left), the solution of \eqref{eq:ip} in this case is severely ill-conditioned. Moreover, the operator is not of convolutional form, and the ill-conditioning increases for signal components corresponding to later times. This can be seen in the right image in Figure~\ref{fig:A}, where the right part of the signal is significantly more blurred than the left.

\paragraph{Network architecture:} As the network architecture \((\nn_\theta)_{\theta \in \Theta}\), we utilize a one-dimensional version of the 2D U-Net \cite{ronneberger2015u} with  skip connection. The architecture begins with several convolutional blocks, each consisting of two 1D convolutional layers with ELU activation, where the number of filters doubles at each block, starting from 16. Downsampling is achieved using strided convolutions, which halve the spatial size of the feature maps at each stage. Upon reaching the bottleneck, the network employs transposed convolutions (upsampling) to restore the original spatial size. Skip connections between corresponding layers in the encoder and decoder paths are incorporated to retain detailed features. The final output is generated through a concluding convolutional layer, with the option to either add the input signal to the output (thus learning the residuals) or produce the entire signal directly without this addition. Dropout layers are integrated to prevent overfitting and enhance generalization. The design is intentionally kept simple to minimize computational time, as developing a sophisticated  network architecture is not the primary focus of this work.

\paragraph{Network training:} 
For the training signals \(\signal_i\), we utilize a collection of block signals similar to those in~\cite{DoJo94}. We scale the input signals \(x_i\) to have a maximum value of \(1\) and construct noisy data as \(\Ao \signal_i + z_{i,n}\), where \(z_{i,n}\) is normally distributed noise with mean zero and standard deviation proportional to \(\sigma\) and  the mean value of \(\Ao \signal_i\). The full training dataset consists of \(5000\) ground truth signals \(\signal_i\) and \(5000\) perturbed signals for each $n$ and each \(\alpha[n,j] = j / 10\) with \(j \in \{1, \dots, 8\}\). According to Section~\ref{sec:train}, the network is trained by minimizing the risk \(\risk_n( \theta ) = \sum_i  \sum_j \norm{ \nn_\theta(\signal_i + r_{i,n,j}) - \signal_i }^2\). The trained network is then denoted as \(\nn_n \coloneqq \nn_{\theta_n}\), where \(\theta_n\) is the numerical minimizer of \(\risk_n\).

\paragraph{Numerical Morozov regularization:}
Reconstruction is done by numerically solving \eqref{eq:morozov1} with the noise-adaptive data-driven regularizer $\reg_n \coloneqq \norm{(\Id- \nn_n)(\cdot)}^2/2 + \norm{\Lo \signal }_1$, where $\norm{\cdot}_1$ is the $\ell^1$-norm and $\Lo$ is either zero (no additional regularizer) or taken as the discrete central difference operator with Neumann boundary conditions (TV as additional regularizer). For computing numerical solutions we write \eqref{eq:morozov1} in the form 
\begin{equation} \label{eq:morozov2}
   \min_{\signal}  
    \|\nn_n(\signal)-\signal\|_2^2 + \lambda \| \Lo \signal\|_1  + \mathbf{1}_{B} (\Ao \signal)  
\end{equation}
with  $ \mathbf{1}_{B}$ denoting the indicator function of  $B = \{ \data \in \R^d \mid \| \data - \data_\delta \| \le \delta \}$.   Optimization problem \eqref{eq:morozov2} is then solved using the primal dual algorithm of \cite{condat2013primal}. With the abbreviations $f(\signal) := \| \nn_n(\signal ) - \signal \|_2^2$, $h_1 := \la \| \cdot \|_1$ and $h_2 : = \mathbf{1}_{B}$, parameters $\tau, \sigma, \rho > 0$ and initial values $\signal^{(0)}$ and $y^{(0)} =(0,0)$ the proposed reconstruction algorithm reads
\begin{align*} 
z^{(i+1)} &:= x^{(i)} - \tau \nabla f(x^{(i)}) - \tau \Lo^*y_1^{(i)} - \tau \Ao^* y_2^{(i)} 
\\
x^{(i+1)} &:= \rho z^{(i+1)} + (1-\rho)x^{(i)} 
\\
w_1^{(i+1)} &:= \prox_{\sigma h_1^*}(y_1^{(i)} + \sigma \Lo(2 z^{(i+1)} -x^{(i)}))
\\
y_1^{(i+1)} &:= \rho w_1^{(i+1)} + (1- \rho) y_1^{(i)}
\\
w_2^{(i+1)} &:= \prox_{\sigma h_2^*}(y_2^{(i)} + \sigma \Ao(2 z^{(i+1)} -x^{(i)}))
\\
y_2^{(i+1)} &:= \rho w_2^{(i+1)} + (1- \rho) y_2^{(i)} \,.
\end{align*}
Here $\prox$ denotes the proximity mapping and $()^*$ the Fenchel dual. The proximity mappings $\prox_{\sigma h_1^*}$, $\prox_{\sigma h_2^*}$ can be easily computed using the relation $\prox_{h^*} + \prox_{h} = \Id$ and the known expressions for the proximity mappings of $\| \cdot \|_1$ and $\mathbf{1}_{B}$. When \(\mathbf{L} = 0\), the iterative algorithm does not require the \(y_1\) and \(w_1\) updates.

Concerning the convergence of the algorithm, we note that for a linear network \(\nn_n\), the functional \(f\) is convex, and the convergence analysis of \cite[Theorem 3.3]{condat2013primal} can be applied, guaranteeing the convergence of the sequence \((\signal^{(i)}, y^{(i)})_{i \in \mathbb{N}}\). In the general and important case of a non-convex network regularizer, we are not aware of theoretical results that assure the proposed algorithm's convergence. Results similar to those in \cite{valkonen2014primal}, obtained for a related algorithm, seem possible and would be an interesting line of research.

\begin{figure}[htb!]
\begin{center}
\includegraphics[height=0.3\textwidth]{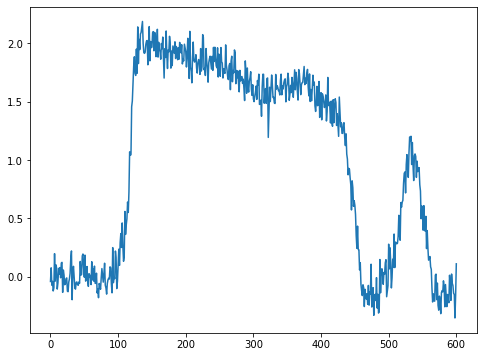}
\includegraphics[height=0.3\textwidth]{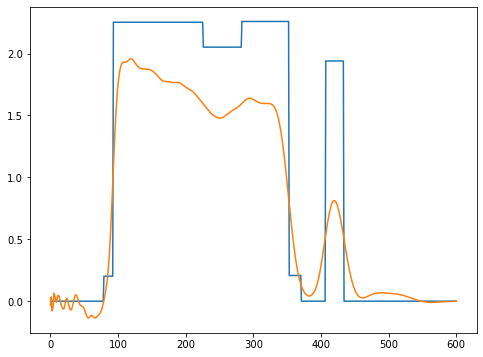}
\includegraphics[height=0.3\textwidth]{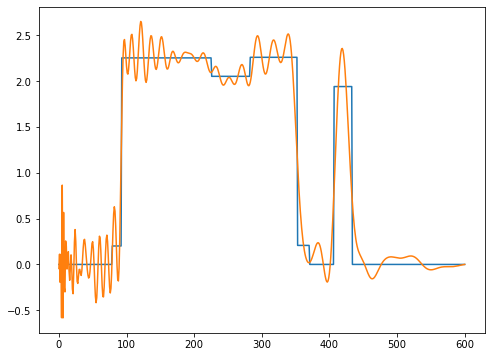}
\includegraphics[height=0.3\textwidth]{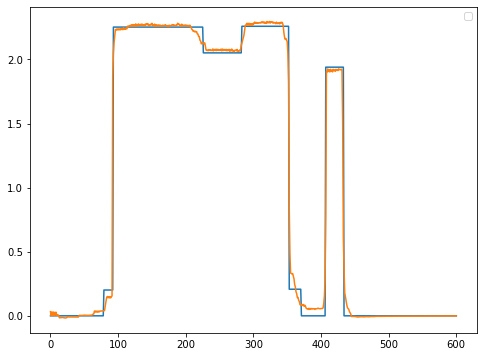}
\caption{Top left: Noisy data $\data_\delta$. Top right:  BP reconstruction. Bottom left:  SVD reconstruction.  Bottom right: DD-Morozov regularization .}
\label{fig:poc}
\end{center}
\end{figure}

\paragraph{Tikhonov regularization:} 
In the results section, we will compare the Morozov-type regularizations with their Tikhonov counterparts, which minimize the Tikhonov functional
\begin{equation} \label{eq:nett}
    \mathcal{T}_{\alpha, \data_\delta}(\signal) := 
    \frac{1}{2} \norm{\Ao \signal - \data_\delta}^2_2 + \alpha \|\nn_n(\signal)-\signal\|_2^2  +  \lambda \| \Lo \signal\|_1 \,.
\end{equation}
The case \(\alpha = 0\) corresponds to standard TV-regularization \cite{acar1994analysis}, and the case \(\lambda = 0\) corresponds to NETT (network Tikhonov) regularization \cite{li2018nett}. Additionally, we present results for the mixed case where both regularizers are active. Minimization of all NETT variants is performed using incremental gradient descent.


\subsection{Results}

\paragraph{Proof of concept:}
Figure~\ref{fig:poc} shows results for a randomly selected block signal $\signal$, which is not part of the training data. The upper left image displays the noisy attenuated signal $\data_\delta$, while the upper right image shows the backprojection (BP) reconstruction $\Ao^T \data_\delta$. The lower left image illustrates the truncated SVD reconstruction $\So_\alpha \data_\delta$ with $\alpha = 0.1$, and the lower right image shows the results using the proposed DD-Morozov regularization with TV as an additional regularizer. The BP reconstruction is noticeably damped, whereas the SVD reconstruction exhibits strong oscillations. The DD-Morozov regularization, as shown by applying methods~\eqref{eq:reg1} and~\eqref{eq:morozov1}, is clearly superior. Similar results have been achieved for other randomly selected test signals.

\begin{table}[htb!]
    \centering
    \begin{tabular}{llll}
        \toprule
                            & TV         & CNN        & TV plus CNN \\
        \midrule
        \multicolumn{4}{l}{$\ell^2$ errors using CNN initialization} \\
        \midrule
        Tikhonov  & $0.0866 \pm 0.045$  & $0.0874 \pm 0.045$  & $0.0866 \pm 0.045$ \\
        Morozov & $0.0877 \pm 0.046$  & $0.0872 \pm 0.045$  & $0.0866 \pm 0.044$ \\
        \midrule 
        \multicolumn{4}{l}{$\ell^2$ errors using zero initialization} \\
        \midrule
        Tikhonov           & $0.100 \pm 0.032$   & $0.107 \pm 0.032$   & $0.100 \pm 0.039$ \\
        Morozov            & $0.090 \pm 0.034$   & $0.108 \pm 0.033$   & $0.103 \pm 0.033$ 
        \\ \bottomrule
    \end{tabular}
    \caption{Comparison of  various variants of Morozov regularization with variants of Tikhonov regularization. CNN-Tikhonov  is the NETT and  TV-Tikhonov is standard TV-regularization. Mean and standard deviation of the  $\ell^2$ reconstruction errors have been computed for 500 samples with  $\delta = 0.1$.}
    \label{tab:comparison}
\end{table}

\paragraph{Comparison with NETT  and TV regularization:}
\label{subsec:comp}

While the results shown in Figure~\ref{fig:poc} demonstrate that DD-Morozov regularization improves reconstructions compared to very simple regularization methods, we next evaluate DD-Morozov regularization with TV as an additional regularizer against DD-Morozov regularization without additional regularization, Morozov with TV, as well as their counterparts minimizing the Tikhonov functional \eqref{eq:nett}. For all methods, we evaluated performance using either an initial estimate from the neural network prediction with truncated SVD as input or zero initialization.  The results shown in Table~\ref{tab:comparison} clearly demonstrate that the network initialization consistently outperforms zero initialization. Additionally, all methods with truncated SVD initialization reduced the mean error  \(0.089 \pm 0.046\) of the initial guess.  The TV-Tikhonov method performed similarly to NETT and NET-Morozov.  Figure~\ref{fig:comparison} illustrates a representative example of signal reconstruction, where the proposed method with TV regularization strikes a good balance between smoothing the signal and preserving fine details.

\begin{figure}[htb!]
    \centering
    \includegraphics[width=\textwidth]{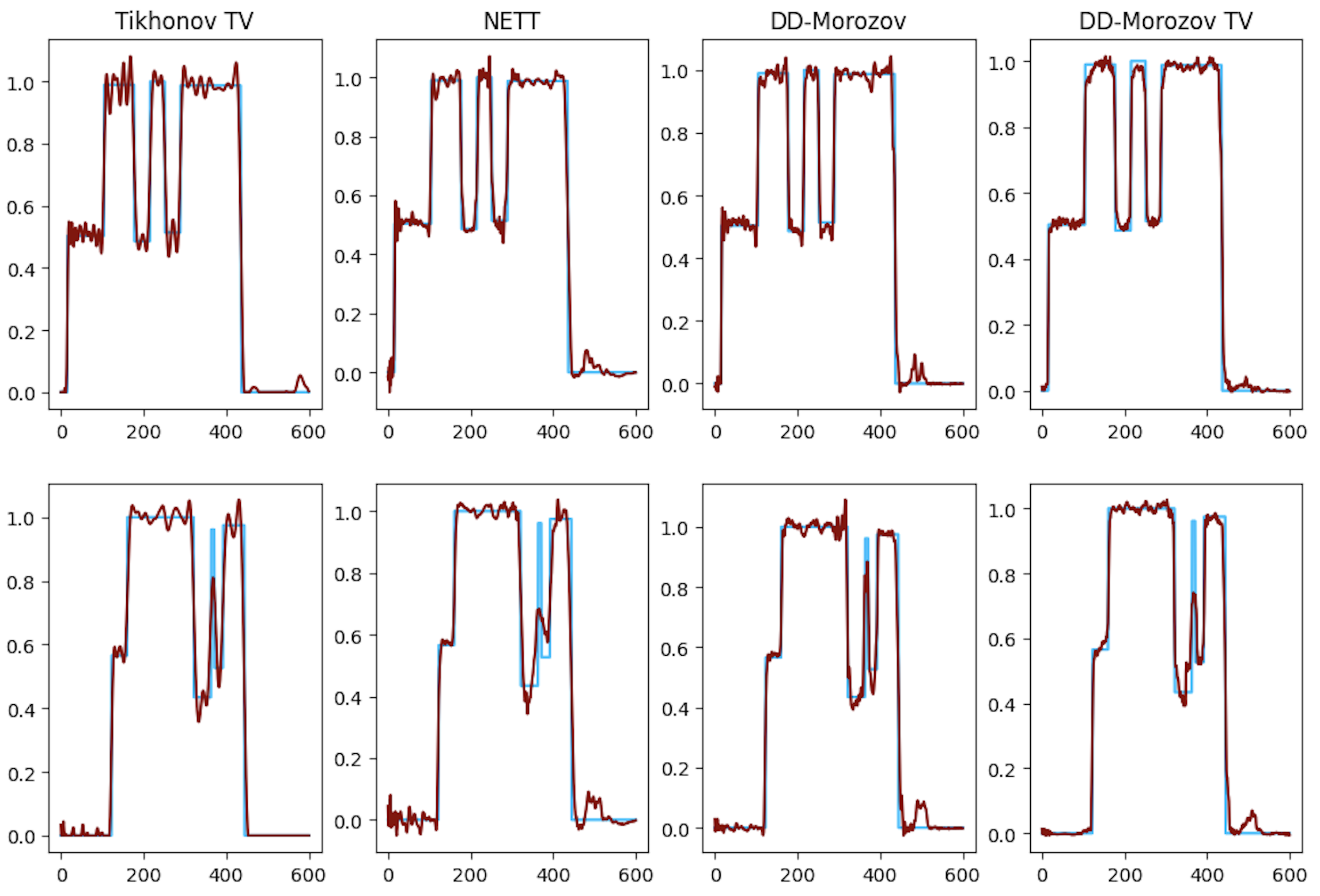}
    \caption{Comparison of several methods using either the network output (top) or the zero signal as initial guess (bottom). The columns from left to right show Tikhonov regularization with TV as the regularizer, the NETT, DD-Morozov without TV, and DD-Morozov with TV as an additional regularizer. For the network as initial guess, the truncated SVD reconstruction applied to noisy data has been used as input for the neural network that has been trained for the regularizer.
}
    \label{fig:comparison}
\end{figure}

\paragraph{Numerical convergence:}

To illustrate  the convergence  stated in Theorem \ref{thm:weak}, we performed numerical simulations under decreasing noise levels. Specifically, we trained six different networks, each exposed to a separate  noise level  with $\sigma \in \{ 0.005, 0.01, 0.05, 0.1, 0.15, 0.2 \}$. We  analyzed the method's average performance across $500$ additional samples while progressively decreasing the noise levels. Our empirical results, shown in Figure~\ref{fig:convergence}, indicate that as the noise levels decrease, the method's performance improves, thus providing empirical support for the theoretical convergence result presented in Theorem~\ref{thm:weak}.

\begin{figure}[htb!]
\centering
\includegraphics[width=0.7\textwidth]{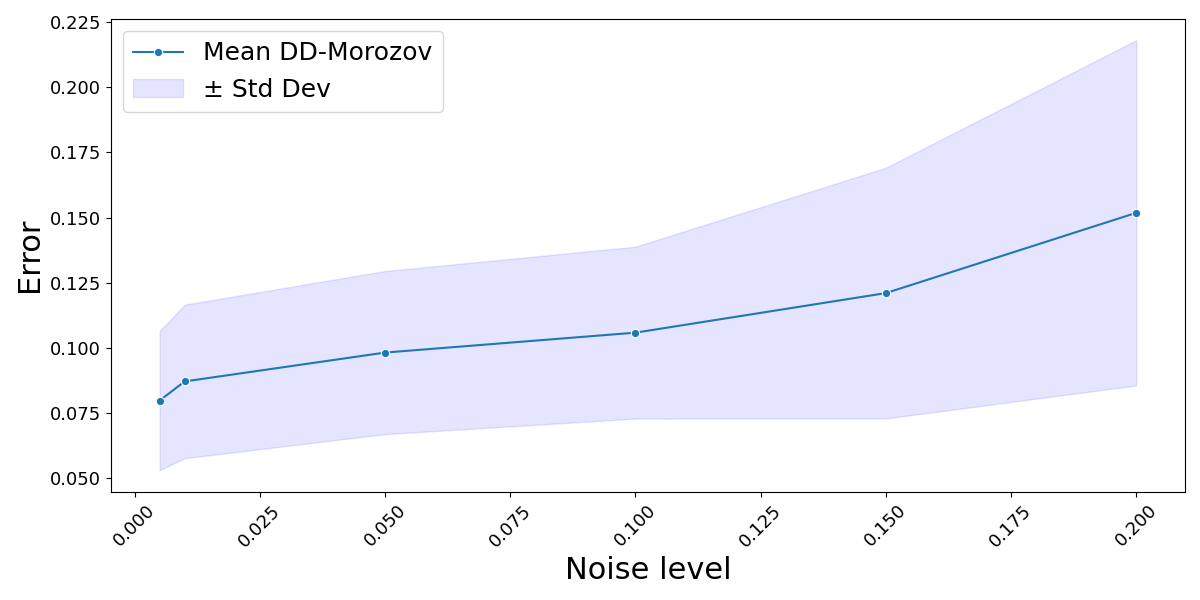}
\caption{Mean value and standard deviation of the $\ell^2$ reconstruction error as a function of the noise level $\delta$.}
\label{fig:convergence}
\end{figure}

\paragraph{Influence of correct noise level:}
Our proposed training strategy requires an assumed noise level. To assess the sensitivity of the trained networks to variations in this levels, we conducted experiments comparing the performance of networks trained at noise levels of $0.01$ and $0.2$ when evaluated on a test dataset with a noise level of $0.1$. The results are shown in Table \ref{tab:noise}. We observe that training the networks with incorrect noise levels leads to a slight decrease in performance. However, the results appear relatively stable, even when the wrong noise level is used for training.

\begin{table}[htb!]
    \centering
    \begin{tabular}{ll}
        \toprule
        Noise level in training & Average $\pm$ standard deviation of $\ell^2$-error \\
        \midrule
        $0.10$ (correct)         & $0.110 \pm 0.033$  \\
        $0.01$ (underestimated) & $0.114 \pm 0.033$  \\
        $0.20$ (overestimated)   & $0.112 \pm 0.032$  \\
        \bottomrule
    \end{tabular}
    \caption{Average reconstruction error and standard deviation of DD-Morozov regularization on the dataset with noise level $0.1$ using networks trained on correct (top)  underestimated (middle; trained with 0.01) and overestimated (bottom; trained with 0.2) noise level.}
    \label{tab:noise}
\end{table}

\paragraph{Theory versus numerics:} The weak convergence result from Theorem~\ref{thm:weak} relies on networks that satisfy the conditions of Assumption~\ref{ass:well}, which appear to be quite natural. Due to the architectural choices, the networks \(\nn\) and \(\nn_{\theta(\delta)}\) are Lipschitz continuous, fulfilling \ref{a1}. Assumptions \ref{a2} and \ref{a2a} are conditions on the  convergence of \(\nn_{\theta(\delta)}\) to some network \(\nn\) as \(\delta \to 0\). Since we use the same architecture and apply perturbations that diminish with decreasing noise levels this behavior is anticipated in our case; a rigorous  proof however still requires careful adjustments to the loss function and proper initialization. This could be an interesting direction for future research, particularly as it applies to trained regularizers in general. Considering that the simulations are carried out in a finite-dimensional, ill-conditioned setting, strong convergence is already implied by weak convergence. However, ensuring that convergence is independent of the discretization requires total convexity (see Theorem~\ref{thm:strong}). In a broader context, a trained regularizer may not always satisfy this assumption. Determining how to guarantee this during training presents another interesting  topic for future research.

\section{Summary}
\label{sec:conclusion}

In this paper, we introduced and analyzed neural network-based noise-adaptive Morozov regularization using a data-driven regularizer (NN-Morozov regularization). We performed a complete convergence analysis that also allows for noise-dependent regularizers. In addition, we established convergence in strong topology. To make our approach practical, we developed a simple yet efficient training strategy extending NETT \cite{li2018nett}. We verified our methodology through numerical experiments, with a special focus on its application to attenuation correction for PAT.  Our research can provide the basis for a broader integration of data-driven regularizers into various variational regularization techniques.

\end{document}